\documentclass[reqno]{amsart}
\usepackage{amsmath,amsfonts,amssymb}
\usepackage{color}

\begin{document}

\title[On Wigner's theorem]{On Wigner's theorem in smooth normed spaces}

\author{Dijana Ili\v{s}evi\'{c}}

\address{Department of Mathematics,
University of Zagreb, Bijeni\v{c}ka 30, P.O. Box 335, 10002 Zagreb,
Croatia}

\email{ilisevic@math.hr}

\author{Aleksej Turn\v{s}ek}

\address{Faculty  of Maritime Studies and Transport, University of Ljubljana, Pot pomor\-\v{s}\v{c}akov 4, 6320 Portoro\v{z}, Slovenia and Institute of Mathematics, Physics and Mechanics, Jadranska 19, 1000 Ljubljana, Slovenia}

\email{aleksej.turnsek@fmf.uni-lj.si}
\thanks{This research was supported in part by the Ministry of Science
and Education of Slovenia.}





\begin{abstract}
In this note we generalize the well-known Wigner's unitary-anti\-unitary theorem.
For $X$ and $Y$ smooth normed spaces and $f:X\to Y$ a surjective mapping such that $|[f(x),f(y)]|=|[x,y]|$, $x,y\in X$, where $[\cdot,\cdot]$ is the unique semi-inner product, we show that $f$ is phase equivalent to either a linear or an anti-linear surjective isometry. When $X$ and $Y$ are smooth real normed spaces and $Y$ strictly convex, we show that Wigner's theorem is equivalent to $\{\|f(x)+f(y)\|,\|f(x)-f(y)\|\}=\{\|x+y\|,\|x-y\|\}$, $x,y\in X$.
\end{abstract}

\maketitle

\newtheorem{theorem}{Theorem}[section]
\newtheorem{proposition}[theorem]{Proposition}
\newtheorem{lemma}[theorem]{Lemma}
\newtheorem{corollary}[theorem]{Corollary}
\newtheorem{definition}{Definition}
\theoremstyle{definition}
\newtheorem{example}[theorem]{Example}
\newtheorem{xca}[theorem]{Exercise}
\newtheorem{question}{Question}

\theoremstyle{remark}
\newtheorem{remark}{Remark}[section]

\section{Introduction}

Let $(H,(\cdot,\cdot))$ and $(K,(\cdot,\cdot))$ be inner product spaces over $\mathbb F\in\{\mathbb R,\mathbb C\}$ and suppose that $f:H\to K$ is a mapping satisfying 
\begin{equation}| (f(x),f(y))|=|(x,y)|,\quad x,y\in H.\end{equation}
Then the famous Wigner's theorem says that $f$ is a solution of (1) if and only if it is phase equivalent to a linear or an anti-linear isometry, say $U$, that is, $$f(x)=\sigma(x) Ux,\quad x\in H,$$ 
where $\sigma: H\to\mathbb F$, $|\sigma(x)|=1$, $x\in H$, is a so called phase function. This celebrated result plays a very important role in quantum mechanics and in representation theory in physics.
There are several proofs of this result, see \cite{Bargmann, Freed, Geher, Gyory, Lomont, Ratz, Sharma1, Sharma3} to list just some of them. For generalizations to Hilbert $C^*$-modules see \cite{Bakic, Molnar}.

On each normed space $X$ over $\mathbb{F}$ there exists at least one  semi-inner product (s.i.p.), see \cite{Giles, Lumer}, on $X$ which is a function $[\, \cdot, \cdot \,] \colon X\times X\to\mathbb{F}$ with the following properties:
\begin{enumerate}
\item $[x+y,z]=[x,z]+[y,z]$, $[\lambda x,y]=\lambda[x,y]$, $[x,\lambda y]=\overline{\lambda}[x,y]$ for all $\lambda\in\mathbb{F}$ and $x,y \in X$,
\item $[x,x]>0$ when $x\ne0$,
\item $|[x,y]|\leq[x,x]^{1/2}[y,y]^{1/2}$ for all $x,y \in X,$
\end{enumerate}
and moreover, it is compatible with the norm in the sense that $[x,x]^{1/2}=\|x\|$.

Recall that $X$ is said to have a Gateaux differentiable norm at $x\ne0$ whenever  
$$\lim_{t\to0,t\in\mathbb{R}}\frac{\|x+ty\|-\|x\|}{t}$$
exists for all $y\in X$.

Remember also that a support functional $\phi_x$ at $x\in X$ is a
norm-one linear functional in $X^*$ such that $\phi_x(x) = \|x\|$. By
the Hahn--Banach theorem there always exists at least one such
functional for every $x\in X$.

A normed space $X$ is said to be smooth at $x$ if there
exists a unique support functional at $x$.
If $X$ is smooth at each one of its points then $X$ is said to be
smooth. It is well known, see for instance \cite[Theorem 1, p.~22]{Diestel}, that a Banach space $X$ is smooth at $x$ if and only if the norm is Gateaux differentiable at $x$. Moreover, in this case,  the real part $\text{Re}\,\phi_x$ of a unique support functional $\phi_x$ at $x$ is given by
\begin{equation}\label{smooth}
\text{Re}\,\phi_x(y)=\lim_{t\to0, t\in\mathbb R}\frac{\|x+ty\|-\|x\|}{t}.
\end{equation}

If $X$ is not smooth then there are many semi-inner products compatible with the norm. However, if $X$ is smooth then $[x,y]:=\|y\|\phi_y(x)$, where $\phi_y$ is the support functional at $y$, is the unique semi-inner product with $[x,x]^{1/2}=\|x\|$.

Now the following natural question arises: Let $X,Y$ be normed spaces and $f:X\to Y$ a mapping such that
\begin{equation}\label{normedWigner}
|[f(x),f(y)]|=|[x,y]|,\quad x,y\in X.
\end{equation}
Is it true that $f$ satisfies (\ref{normedWigner}) if and only if it is phase equivalent to either a linear or an anti-linear isometry? Let us first check that in general even not all linear isometries satisfy (\ref{normedWigner}).

\begin{example}
Let $T \colon (l_\infty^2,\mathbb R)\to(\l_\infty^2,\mathbb R)$ be defined by $T(x,y)=(y,x)$ and let the semi-inner product for $x=(x_1,x_2)$ and $y=(y_1,y_2)$ be defined by
$$[x,y]=\begin{cases}
x_1y_1&\text{if}\quad |y_1|>|y_2|\\
x_2y_2&\text{if}\quad|y_1|<|y_2|\\
\frac{3}{4}x_1y_1+\frac{1}{4}x_2y_2&\text{if}\quad |y_1|=|y_2|.
\end{cases}
$$
Then for $x=(1,0)$ and $y=(1,1)$ we get $[x,y]=\frac{3}{4}$ and $[Tx,Ty]=\frac{1}{4}$.
\end{example}

However, if $X$ and $Y$ are smooth normed spaces, then a mapping phase equivalent to a linear or an anti-linear isometry satisfies (\ref{normedWigner}). Indeed, if $U$ is a linear or an anti-linear isometry, then $\|Uy+tUx\|=\|y+tx\|$, $t\in\mathbb R$, hence by (\ref{smooth})
$$\text{Re}\,\phi_{Uy}(Ux)=\text{Re}\,\phi_y(x)$$
and then also $[Ux,Uy]=[x,y]$. From 
$$[f(x),f(y)]=[\sigma(x)Ux,\sigma(y)Uy]=\sigma(x)\overline{\sigma(y)}[Ux,Uy]$$
the claim follows. In our main result Theorem \ref{main} we show that the converse also holds.



\section{Results}
Throughout, for a normed space $(X, \Vert  \cdot  \Vert)$, by $[\, \cdot, \cdot \,]$ we denote a semi-inner product satisfying $\Vert x \Vert = [x,x]^{1/2}$. We denote by $\mathbb PX=\{\langle x\rangle : x \in X\}$ the set of all one-dimensional subspaces of
a normed space $X$. If $M\subset X$ then $\langle M\rangle$ will denote the subspace generated by the set $M$. If $L\subseteq X$ is a two-dimensional subspace then $L=\langle L\rangle$ is called a projective line. Recall also that $A:X\to Y$ is semilinear if $A(x+y)=Ax+Ay$ and $A(\lambda x)=h(\lambda)Ax$, $x,y\in X$, $\lambda \in\mathbb{F}$, where $h:\mathbb{F}\to\mathbb{F}$ is a homomorphism. Next we state the fundametal theorem of projective geometry in the form in which it will be needed, see \cite[Theorem 3.1]{Faure}.  

\begin{theorem}[Fundamental theorem of projective geometry]\label{projective}
Let $X$ and $Y$ be vector spaces over $\mathbb{F}$ of dimensions at least three. Let $g: \mathbb PX\to\mathbb PY$ be a mapping such that
\begin{itemize}
\item[(i)] The image of $g$ is not contained in a projective line.
\item[(ii)] $0\ne c\in \langle a,b\rangle, a\ne 0\ne b,$ implies $g(\langle c\rangle)\in \langle g(\langle a\rangle), g(\langle b\rangle)\rangle$.
\end{itemize}
Then there exists an injective semilinear mapping $A:X\to Y$ such that
$$g(\langle x\rangle)=\langle Ax\rangle,\quad 0\ne x\in X.$$
Moreover, $A$ is unique up to a non-zero scalar factor.
\end{theorem}

In the proof of the next theorem we will also need the notion of orthogonality in normed spaces. Remember that $x\in X$ is Birkhoff-James orthogonal to $y\in X$, 
$$x\perp y\quad \text{if}\quad \|x+\lambda y\|\geq\|x\|\quad \text{for all }\lambda\in\mathbb{F}.$$

When $x\in X$ is a point of smoothness, then $x\perp y$
if and only if $y$ belongs to the kernel of the unique support
functional at~$x$, see \cite[Proposition 1.4.4.]{Fleming-Jamison}. 
Important consequence is that Birkhoff-James orthogonality is right additive in smooth spaces, that is, $x\perp y, x\perp z\Rightarrow x\perp y+z$.
Also note that in this case $x\perp y$ if and only if $[y,x]=0$.

\begin{theorem}\label{main}
Let $X$ and $Y$ be smooth normed spaces over $\mathbb{F}$ and suppose that $f \colon X\to Y$ is a surjective mapping satisfying 
$$|[f(x),f(y)]|=|[x,y]|,\quad x,y\in X.$$
\begin{itemize}
\item[(i)] If $\dim X\geq2$ and $\mathbb{F}=\mathbb{R}$, then $f$ is phase equivalent to a linear surjective isometry.
\item[(ii)] If $\dim X\geq2$ and $\mathbb{F}=\mathbb{C}$, then $f$ is phase equivalent to a linear or conjugate linear surjective isometry.
\end{itemize}
\end{theorem}
\begin{proof}
Let $\lambda\in\mathbb{F}$ and $x\in X$. We will show that $f(\lambda x)=\gamma f(x)$, where $\gamma=\gamma(\lambda,x)$ depends on $\lambda$ and on $x$, and $|\gamma|=|\lambda|$. The function  
\begin{equation}\label{min}
\xi\mapsto\|f(\lambda x)-\xi f(x)\|
\end{equation} 
is continuous and tends to infinity when $|\xi|$ tends to infinity. Hence there is at least one point, say $\gamma$, such that the function in (\ref{min})  achieves its global minimum.  Thus
$$\min_{\xi\in\mathbb{F}}\|f(\lambda x)-\xi f(x)\|=\|f(\lambda x)-\gamma f(x)\|.$$
Note that  
$$\|f(\lambda x)-\gamma f(x)+\mu f(x)\|\geq\|f(\lambda x)-\gamma f(x)\|$$
for all $\mu\in\mathbb{F}$, hence $f(\lambda x)-\gamma f(x)\perp f(x)$. Since $f$ is surjective, there is $z\in X$ such that $f(z)=f(\lambda x)-\gamma f(x)$. Then from $f(z)\perp f(x)$ we get $z\perp x$, and then $z\perp \lambda x$ and $f(z)\perp f(\lambda x)$. Since $Y$ is smooth, Birkhoff-James orthogonality is right additive, so from $f(z)\perp f(\lambda x)$ and $f(z)\perp f(x)$ we conclude $f(z)\perp f(\lambda x)-\gamma f(x)=f(z)$. Thus $f(z)=0$ and we have $f(\lambda x)=\gamma f(x)$. 
Furthermore,
$$\vert \lambda \vert \Vert x \Vert = \|\lambda x\|=\|f(\lambda x)\|=\|\gamma f(x)\|=\vert \gamma \vert \Vert f(x) \Vert = \vert \gamma \vert \Vert x \Vert,$$
which implies $|\gamma|=|\lambda|$. 

Next, let $x,y\in X$ be linearly independent. We will show that
$f(x+y)=\alpha f(x)+\beta f(y)$, where $\alpha=\alpha(x,y)$, $\beta=\beta(x,y)$, and $|\alpha|=|\beta|=1$. Analogously  as before we obtain $\alpha, \beta\in\mathbb F$ such that
$$\min_{\xi,\eta\in\mathbb F}\|f(x+y)-\xi f(x)-\eta f(y)\|=\|f(x+y)-\alpha f(x)-\beta f(y)\|.$$
Furthermore, it is easy to see that 
$$f(x+y)-\alpha f(x)-\beta f(y)\perp f(x)\quad\text{and}\quad f(x+y)-\alpha f(x)-\beta f(y)\perp f(y).$$
Take $z\in X$ such that $f(z)=f(x+y)-\alpha f(x)-\beta f(y)$. Then $f(z)\perp f(x)$ implies $z\perp x$, $f(z)\perp f(y)$ implies $z\perp y$ and smoothness of $X$ implies $z\perp x+y$ and then $f(z)\perp f(x+y)$. Hence $f(z)\perp f(z)$ and $f(z)=0$. Let us show that $|\alpha|=1$. Let $\min_\lambda\|x+\lambda y\|=\|x+\lambda_0y\|$. Then $x+\lambda_0y\perp y$ and $x+\lambda_0y\not\perp x$. Indeed, suppose that $x+\lambda_0y\perp x$. Then by the right additivity we get $x+\lambda_0y\perp x+\lambda_0y$. This would mean that $x+\lambda_0y=0$, a contradiction because $x$ and $y$ are linearly independent. 
Denote $w=x+\lambda_0y$. 
	Since $w \perp y$ we also have $f(w) \perp f(y)$.
	Then
	\begin{eqnarray}
	[f(x+y), f(w)]&=&\alpha[f(x), f(w)]+\beta[f(y), f(w)]\nonumber \\
	&=&\alpha[f(x), f(w)],\nonumber
	\end{eqnarray}
	which implies
	\begin{eqnarray}
	\vert \alpha \vert \vert[x, w]\vert &=&\vert \alpha \vert \vert[f(x), f(w)]\vert=\vert[f(x+y), f(w)]\vert\nonumber \\
	&=&\vert[x+y, w]\vert = \vert[x, w]\vert,\nonumber
	\end{eqnarray}
hence $|\alpha|=1$. Similarly we get $|\beta|=1$.

Let us prove that $f$ induces a surjective mapping $\tilde{f} \colon \mathbb{P}X\to\mathbb{P}Y$ defined by $\tilde{f}(\langle x\rangle)=\langle f(x)\rangle$.  Suppose $\langle x\rangle=\langle y\rangle$, that is $y=\lambda x$. Then $f(y)=f(\lambda x)=\gamma f(x)$ for some $\gamma\in\mathbb F$ and then $\langle f(y)\rangle=\langle f(x)\rangle$. So $\tilde{f}$ is well defined and  surjective because $f$ is surjective.

Now suppose that $\dim X\geq3$ and let $x\in X$ be a unit vector. Choose a unit vector $y\in\ker\phi_x$, where $\phi_x$ is the support functional at $x$, and then choose a unit vector $z\in\ker\phi_x\cap\ker\phi_y$, where $\phi_y$ is the support functional at $y$. Then from $x\perp y$, $x\perp z$ and $y\perp z$ follows that $x,y,z$ are linearly independent. Indeed, $y$ and $z$ are linearly independent because $y\perp z$. From $x\perp y$ and $x\perp z$ it follows, using homogeneity and right additivity of Birkhoff-James orthogonality, that $x\perp\langle y,z\rangle$, hence $x,y,z$ are linearly independent. Now $f(x), f(y), f(z)$ are unit vectors such that $f(x)\perp f(y)$, $f(x)\perp f(z)$ and $f(y)\perp f(z)$. As before we conclude that $f(x), f(y)$ and $f(z)$ are linearly independent. So the image of $f$ is not contained in a two-dimensional subspace, thus the image of $\tilde{f}$ is not contained in a projective line. This shows that  $\tilde{f}$ satisfies condition (i) of Theorem \ref{projective}. Furthermore, from $f(\lambda x)=\gamma f(x)$ and $f(x+y)=\alpha f(x)+\beta f(y)$ it follows that condition (ii) of Theorem \ref{projective} is also satisfied.

Thus by Theorem \ref{projective} we conclude that $\tilde{f}$ is induced by a bijective semilinear mapping $A:X\to Y$, that is, 
$$\tilde{f}(\langle x\rangle)=\langle Ax\rangle,\quad x\in X.$$
Fix a nonzero $x\in X$. Then $f(x)=\lambda Ax$ for some nonzero $\lambda\in\mathbb{F}$. Let $y\in X$ be such that $x$ and $y$ are linearly independent. Then $f(y)=\mu Ay$ and $f(x+y)=\nu A(x+y)$. Note also that $Ax$ and $Ay$ are linearly independent since $A$ is semilinear and bijective.
Thus from $f(x+y)=\alpha f(x)+\beta f(y)=\alpha\lambda Ax+\beta\mu Ay$ we get
$\alpha\lambda=\nu$ and $\beta\mu=\nu$. Since $|\alpha|=|\beta|=1$ we get $|\lambda|=|\mu|=|\nu|$. 
Hence $f(z)=\lambda(z)Az$ with $\vert \lambda(z) \vert=\vert \lambda \vert$ for all $z\in X$.
Let $U=\lambda A$ and $\sigma(z)=\lambda(z)/\lambda$ for every $z \in X$.
	Then $\sigma \colon X \to \mathbb{F}$ is a phase function and
	$$f(z)=\lambda(z)Az=\sigma(z)Uz, \quad z \in X.$$
If $\mathbb{F}=\mathbb{R}$ then $A$ (hence also $U$) is linear, because any nontrivial homomorphism $h \colon \mathbb{R}\to\mathbb{R}$ is identity. Suppose $\mathbb{F}=\mathbb{C}$.
Let $\xi\in\mathbb{C}$. Then
$$f(\xi z)=\lambda(\xi z)A(\xi z)=\lambda(\xi z)h(\xi)Az,$$
and on the other hand $f(\xi z)=\xi'f(z)=\xi'\lambda(z)Az$. Because $|\lambda(\xi z)|=|\lambda(z)|$ and $|\xi'|=|\xi|$ we get $|h(\xi)|=|\xi|$. 
Then $h$ is continuous at zero, hence continuous everywhere.
A continuous homomorphism $h \colon \mathbb{C}\to\mathbb{C}$ is either identity or conjugation.
Therefore $A$, and also $U$, is linear or conjugate linear.
It is now clear that $U$ is an isometry. It is surjective because $f$ is surjective.
This completes the proof.

Let us now suppose that $\dim{X}=2$.
	Let us fix linearly independent $x_0, y_0 \in X$.
	Let $A(x_0)=f(x_0)$.
	For every $\mu \in \mathbb{F}$ there exist $\omega_1, \omega_2 \in \mathbb{F}$ such that 
	$f(x_0+\mu y_0)=\omega_1 f(x_0)+\omega_2 f(y_0)$, 
	with $\vert \omega_1 \vert =1$, $\vert \omega_2 \vert = \vert \mu \vert$.
	Let $h(\mu)=\omega_2 / \omega_1$ and $A(\mu y_0) = h(\mu)f(y_0)$.
	Note that $\vert h(\mu) \vert = \vert \mu \vert$.
	Furthermore, let us define $A(x_0+\mu y_0)=A(x_0)+A(\mu y_0)$.
	For $\lambda, \mu \in \mathbb{F}$,
	$$f(x_0+(\lambda+\mu)y_0)=\omega_1 f(x_0) + \omega_1 h(\lambda+\mu)f(y_0),$$
	and also
	\begin{eqnarray}
	f(x_0+(\lambda+\mu)y_0) &=& f((x_0+\lambda y_0)+\mu y_0) = \omega_2 f(x_0+\lambda y_0) + \omega_3 f(y_0) \nonumber \\
	&=& \omega_4 f(x_0)+ \omega_4 h(\lambda)f(y_0)+\omega_3f(y_0). \nonumber
	\end{eqnarray}
	Since $f(x_0)$ and $f(y_0)$ are also linearly independent, $\omega_4=\omega_1$ and $\omega_4 h(\lambda)+\omega_3 = \omega_1 h(\lambda + \mu)$, with $\vert \omega_1 \vert = 1$ and $\vert \omega_3 \vert = \vert \mu \vert$.
	Then
\begin{eqnarray}\label{0A}
h(\lambda + \mu)=h(\lambda) + \frac{\omega_3}{\omega_1},
\end{eqnarray}
	which implies
	$$\vert \lambda + \mu \vert = \vert h(\lambda+\mu) \vert = \Big{\vert} h(\lambda) + \frac{\omega_3}{\omega_1} \Big{\vert}$$
	with $\vert h(\lambda) \vert = \vert \lambda \vert$ and $\vert \omega_3 / \omega_1 \vert = \vert \mu \vert$.
	This yields
	$$\Big{\vert} \frac{\lambda}{\mu}+1  \Big{\vert} =  \Big{\vert} h(\lambda) \frac{\omega_1}{\omega_3}+1  \Big{\vert}$$
	with $\vert \frac{\lambda}{\mu} \vert = \vert h(\lambda) \frac{\omega_1}{\omega_3} \vert$.
	It can be easily verified that 
	$$\frac{\lambda}{\mu} =  h(\lambda) \frac{\omega_1}{\omega_3} \quad \textup{or} \quad \frac{\overline{\lambda}}{\overline{\mu}} =  h(\lambda) \frac{\omega_1}{\omega_3},$$ that is,
	\begin{eqnarray}\label{1A}
	\frac{\omega_3}{\omega_1}=h(\lambda)\frac{\mu}{\lambda}
	\end{eqnarray}
	or
	\begin{eqnarray}\label{2A}
	\frac{\omega_3}{\omega_1}=h(\lambda)\frac{\overline{\mu}}{\overline{\lambda}}.
	\end{eqnarray}
Let us fix $\eta \in \mathbb{F}$.
If \eqref{1A} holds for $\lambda=1$ and $\mu=\eta-1$ then \eqref{0A} implies
$$h(\eta)=h(1)\eta.$$
If \eqref{2A} holds for $\lambda=1$ and $\mu=\eta-1$ then \eqref{0A} implies
$$h(\eta)=h(1)\overline{\eta}.$$
If $\mathbb{F}=\mathbb{R}$ we are done. Suppose that $\mathbb{F} = \mathbb{C}$.
Note that \eqref{0A} becomes
$$h(\lambda+\mu)=h(\lambda)+h(\lambda)\frac{\mu}{\lambda},$$
or
$$h(\lambda+\mu)=h(\lambda)+h(\lambda)\frac{\overline{\mu}}{\overline{\lambda}}.$$
If for some $\lambda \in \mathbb{F} \setminus \mathbb{R}$ we have $h(\lambda)=h(1)\lambda$ and for some $\mu \in \mathbb{F} \setminus \mathbb{R}$ we have $h(\mu)=h(1)\overline{\mu}$ then 
$$h(\mu)=h(\lambda+(\mu-\lambda))=h(\lambda)+h(\lambda)\frac{\mu-\lambda}{\lambda}=h(1)\mu$$
or
$$h(\mu)=h(\lambda+(\mu-\lambda))=h(\lambda)+h(\lambda)\frac{\overline{\mu}-\overline{\lambda}}{\overline{\lambda}}=h(1)\frac{\lambda}{\overline{\lambda}}\mu.$$
In both cases we arrive at a contradiction with $\lambda, \mu \notin \mathbb{R}$.
Hence $h(\lambda)=h(1)\lambda$ for every $\lambda \in \mathbb{R}$ or  $h(\lambda)=h(1)\overline{\lambda}$ for every $\lambda \in \mathbb{R}$.
Let $k=h(1)$ and let $A(y_0)=kf(y_0)$.
Then $A(\mu y_0)=\mu A(y_0)$ or $\overline{\mu}A(y_0)$, and $A(x_0+\mu y_0)=A(x_0)+\mu A(y_0)$ or $A(x_0+\mu y_0)=A(x_0)+\overline{\mu}A(y_0)$, respectively.
In the first case we extend $A$ to $X$ by $A(\lambda x_0+\mu y_0)= \lambda A(x_0+\frac{\mu}{\lambda}y_0)$, and in the second case by $\overline{\lambda} A(x_0+\frac{\mu}{\lambda}y_0)$.
Such $A$ is linear or conjugate linear.
From 
\begin{eqnarray}
\Vert \lambda x_0 + \mu y_0 \Vert &=& \vert \lambda \vert \, \Vert f(x_0+\frac{\mu}{\lambda} y_0) \Vert = \vert \lambda \vert \, \Vert f(x_0)+h(\frac{\mu}{\lambda})f(y_0) \Vert \nonumber \\
&=& \vert \lambda \vert \, \Vert A(x_0+\frac{\mu}{\lambda} y_0) \Vert = \Vert A(\lambda x_0+\mu y_0) \Vert, \nonumber 
\end{eqnarray}
we conclude that $A$ is an isometry.
Finally,
\begin{eqnarray}
f(\lambda x_0 + \mu y_0) &=&\lambda' f(x_0+\frac{\mu}{\lambda} y_0) \nonumber \\
&=& \lambda' (\omega f(x_0) + \omega h(\frac{\mu}{\lambda}) f(y_0)) = \omega\frac{\lambda'}{\lambda} A(\lambda x_0 + \mu y_0) \nonumber 
\end{eqnarray}
for some $\omega, \lambda' \in \mathbb{F}$ such that $\vert \omega \vert = 1$, $\vert \lambda' \vert = \vert \lambda \vert$.
It remains to define $\sigma(\lambda x_0 + \mu y_0) = \omega \frac{\lambda'}{\lambda}$.
\end{proof}

\begin{remark}
If $X$ is one-dimensional then $X$ is obviously smooth. Suppose that $Y$ is a smooth normed space and $f:X\to Y$ a mapping such that $|[f(x),f(y)]|=|[x,y]|$, $x,y\in X$. Let $\lambda\in\mathbb F$ and fix a unit vector $x\in X$. Analogously as in Theorem \ref{main}, we obtain $f(\lambda x)=\gamma f(x)$ for some $\gamma\in\mathbb F$, which depends on $\lambda$, and $|\gamma|=|\lambda|$. Now for $z=\lambda x$ define  phase function $\sigma(z)=\gamma/\lambda$ and define a linear surjective isometry $U:X\to Y$ by $Uz=\lambda f(x)$. Then $f=\sigma U$ and we conclude that $f$ is phase equivalent to a linear surjective isometry.
\end{remark}

Maksa and P\'{a}les, see \cite{Maksa}, showed that for a mapping $f:H\to K$, where $H$ and $K$ are real inner product spaces, Wigner's theorem is equivalent to the requirement that $f$ satisfies the following condition:
\begin{equation}\label{phaseisometry}
\{\|f(x)+f(y)\|,\|f(x)-f(y)\|\}=\{\|x+y\|,\|x-y\|\},\quad x,y\in H.
\end{equation}
They asked for possible generalizations in the setting of real normed spaces, that is, if $X$ and $Y$ are real normed spaces and $f:X\to Y$ a mapping, is it true that $f$ satisfies (\ref{phaseisometry}) if and only if $f$ is phase equivalent to a linear isometry?

Recall that a normed space $X$ is said to be  strictly convex whenever the unit sphere $S_X$ contains no non-trivial line segments, that is, each point of $S_X$ is an extreme point of a unit ball $B_X$.

The following proposition generalizes \cite[Theorem 2 (i) $\Leftrightarrow$ (iv) $\Leftrightarrow$ (v)]{Maksa}.

\begin{proposition}
Let $X$, $Y$ be real smooth normed spaces, $Y$ strictly convex, $f:X\to Y$ surjective. The following assertions are equivalent:
\begin{itemize}
\item[(i)] $|[f(x),f(y)]|=|[x,y]|$, $x,y\in X$.
\item[(ii)] $f$ is phase equivalent to a linear surjective isometry.
\item[(iii)] $\{\|f(x)+f(y)\|,\|f(x)-f(y)\|\}=\{\|x+y\|,\|x-y\|\}$, $x,y\in X$.
\end{itemize}
\end{proposition}
\begin{proof}
(i) $\Rightarrow$ (ii) is Theorem \ref{main}, and (ii) $\Rightarrow$ (iii) is obvious. It remains to prove (iii) $\Rightarrow$ (i). Let $x=y$. Then from $\{2\|f(x)\|,0\}=\{2\|x\|,0\}$ we get $\|f(x)\|=\|x\|$, $x\in X$. Insert $2x$ and $x$ in (iii) to get
$$\{\|f(2x)+f(x)\|,\|f(2x)-f(x)\|\}=\{3\|x\|,\|x\|\},\quad x\in X.$$
Hence for $x\in X$ either $\|f(2x)+f(x)\|=3\|x\|$ or $\|f(2x)+f(x)\|=\|x\|$. 

Suppose that $x\in X$ is such that $\|f(2x)+f(x)\|=3\|x\|$. Then from
$$3\|x\|=\|f(2x)+f(x)\|\leq\|f(2x)\|+\|f(x)\|=3\|x\|$$
and strict convexity of $Y$ we get $f(2x)=2f(x)$. If $\|f(2x)-f(x)\|=3\|x\|$, then, analogously, we get $f(2x)=-2f(x)$. Therefore $f(2x)=\pm 2f(x)$, $x\in X$.  Let $n=2^m$. Then from $f(nx)=\pm nf(x)$, $x\in X$,  we have
\begin{eqnarray*}
n(\|f(x)+\tfrac{1}{n}f(y)\|-\|f(x)\|)= \|\pm f(nx)+f(y)\|-n\|f(x)\|\\
=\|nx\pm y\|-n\|x\|= n(\|x\pm\tfrac{1}{n}y\|-\|x\|), \quad y\in X.
\end{eqnarray*}
Thus $|[f(y),f(x)]|=|[y,x]|$, $x, y\in X$ and the proof is completed.
\end{proof}

In the last part of the paper we consider mappings $f:X\to Y$ satisfying 
\begin{equation}
[f(x),f(y)]=[x,y],\quad x,y\in X.
\end{equation}
Namely, it is easy to see that in the setting of inner product spaces any such mapping is necessarily a linear isometry.

\begin{proposition}\label{isometry}
Let $X$ and $Y$ be  normed spaces and $f \colon X\to Y$ a mapping such that $[f(x),f(y)]=[x,y]$, $x,y\in X$.
\begin{itemize}
\item[(i)] If $f$ is surjective then $f$ is a linear isometry.
\item[(ii)] If $X=Y$ is smooth Banach space then $f$ is a linear isometry.
\end{itemize}
\end{proposition}
\begin{proof}
(i). From
\begin{align*}
[f(\lambda x+\mu y),f(z)]&=[\lambda x+\mu y,z]=\lambda[x,z]+\mu[y,z]\\
&=\lambda[f(x),f(z)]+\mu[f(y),f(z)]=[\lambda f(x)+\mu f(y),f(z)]
\end{align*}
we conclude
\begin{equation}\label{orth}
[f(\lambda x+\mu y)-\lambda f(x)-\mu f(y),f(z)]=0
\end{equation}
for all $x,y,z\in X$ and all $\lambda, \mu\in \mathbb{F}$. Since $f$ is surjective, linearity of $f$ follows.

(ii). The proof is by contradiction. Let us denote $u=f(\lambda x+\mu y)-\lambda f(x)-\mu f(y)$ and suppose that $u\ne 0$. From (\ref{orth}) we get $f(z)\perp u$ for all $z\in X$ and because $X$ is smooth this is equivalent to $\phi_{f(z)}(u)=0$.  Because of the homogeneity  of orthogonality relation we may and do assume that $\|u\|=1$. From
$$\|\phi_u+\xi\phi_{f(z)}\|\geq|\phi_u(u)+\xi\phi_{f(z)}(u)|=|\phi_u(u)|=1=\|\phi_u\|$$
for all $\xi\in\mathbb{R}$ we conclude $\phi_u\perp\phi_{f(z)}$ for all $z\in X$. 
Homogeneity of Birkhoff-James orthogonality implies $\phi_u\perp\xi\phi_{f(z)}$ for all $z\in X$, $\xi \in \mathbb{R}$.
	Furthermore, 
$$\|f(z)\|\phi_{f(z)}(f(w))=[f(w),f(z)]=[w,z]=\|z\|\phi_z(w)$$
shows 
$$\phi_{f(z)}\circ f=\phi_z,\quad z\in X.$$
By the Bishop--Phelps theorem (see \cite{Bishop-Phelps} or a recent survey \cite{Aron-Lomonosov}), for given $\psi \in X^*$ and $\varepsilon > 0$ there exists $\theta \in X^*$, $\Vert \theta \Vert = \Vert \psi \Vert$ and $\Vert \psi - \theta \Vert < \varepsilon$, such that there exists $z \in S_X$ satisfying $\theta(z)=\Vert \theta \Vert$.
Then $\pm\frac{1}{\|\theta\|}\theta$ is the support functional at $z\in S_X$. Thus $\theta=\pm\|\theta\|\phi_z\in\{\xi\phi_z: z\in X, \xi\in\mathbb{R}\}$. Hence $X^*$ is contained in the norm closure of $\{\xi\phi_z: z\in X, \xi\in\mathbb{R}\}$. Since the reverse inclusion is trivial we conclude that $X^*$ is equal to the norm closure of $\{\xi\phi_z: z\in X, \xi\in\mathbb{R}\}$.
Then from 
$$\{\xi\phi_{f(z)}: z\in X, \xi\in\mathbb{R}\}\supseteq \{\xi\phi_{f(z)}\circ f: z\in X, \xi\in\mathbb{R}\}$$
and $\phi_{f(z)}\circ f=\phi_z$ for all $z\in X$ we conclude that $X^*$ is equal to the norm closure of $\{\xi\phi_{f(z)}: z\in X, \xi\in\mathbb{R}\}$. Then $\phi_u\perp \xi\phi_{f(z)}$ for all $z\in X$ and $\xi \in \mathbb{R}$ implies $\phi_u=0$. This shows that our assumption $u\ne 0$ is false and $f$ must be linear. This completes the proof.
\end{proof}

\begin{corollary}
Let $X$ and $Y$ be normed spaces, $X$ smooth and $f:X\to Y$ a mapping. If $f$ is surjective or $X=Y$ then the following assertions are equivalent:
\begin{itemize}
\item[(i)] $[f(x),f(y)]=[x,y]$, $x,y\in X$.
\item[(ii)] $f$ is a linear isometry.
\end{itemize}
\end{corollary}
\begin{proof}
That (i)$\Rightarrow$(ii) follows by Proposition \ref{isometry}. Let us prove (ii)$\Rightarrow$(i). Take arbitrary $u,v\in Y$ and find $x,y\in X$ such that $u=f(x)$ and $v=f(y)$. Then from
$$\frac{1}{t}(\|u+tv\|-\|u\|)=\frac{1}{t}(\|f(x)+tf(y)\|-\|f(x)\|)=\frac{1}{t}(\|x+ty\|-\|x\|)$$ 
it follows that $Y$ is also smooth and $[f(x),f(y)]=[x,y]$.
\end{proof}

\end{document}